\documentclass{article}
\usepackage{geometry}
\usepackage{graphicx,comment,amsmath,amssymb,amsthm,color,caption,subcaption,url,hyperref,enumerate,lineno,mathrsfs,mathtools,authblk,setspace}

\DeclareMathOperator{\Grad}{grad}
\newcommand{\trace}{{\mathrm{Trace}}}
\newcommand{\arc}{{\mathrm{Arc}}}
\newcommand{\ct}[3]{\mathscr{T}_{\mathbf{#1}}^\mathbf{#2}\left(#3\right)}

\newtheorem{lemma}{Lemma}
\newtheorem{definition}{Definition}
\newtheorem{theorem}{Theorem}
\newtheorem{corollary}{Corollary}

\newtheorem{example}{Example}
\newtheorem{remark}{Remark}

\title{Semidefinite Programming Certificates for Synchronization of Kuramoto Oscillators on Arcs\thanks{\small
Presented at the International Interdisciplinary Chaos Symposium on Chaos and Complex Systems (SCCS 2025), Istanbul, Türkiye. A version of this work has been accepted for publication in the conference proceedings and will appear in \textit{\href{https://link.springer.com/book/9783032091000}{Chaos and Complex Systems: Proceedings of the 6th International Interdisciplinary Chaos Symposium}} (Springer Cham).
}}
\author[1]{Swapnil Tripathi}
\author[2]{Mahmut Kudeyt}
\author[3]{Alkım G\"{o}k\c{c}en}
\author[3]{Sava\c{s} \c{S}ahin}
\author[1,*]{\"{O}zkan Karabacak}
\affil[1]{Kadir Has University, Department of Mechatronics Engineering, Istanbul, T\"{u}rkiye}
\affil[2]{I\c{s}ık University, Department of Mathematics, Istanbul, T\"{u}rkiye}
\affil[3]{Izmir Katip \c{C}elebi University, Department of Electrical and Electronic Engineering, Izmir, T\"{u}rkiye}
\affil[*]{Corresponding author}
\date{}

\begin{document}

\maketitle

\begin{abstract}
A class of Kuramoto models with a general coupling function that can be expressed in terms of a finite number of harmonics, each comprising sinusoidal terms, is studied. We propose a novel approach for certifying local phase synchronization in this class for all initial conditions lying on an arc. The trace parametrization property and Gram matrix representation of a trigonometric polynomial are utilized along with Putinar's Positivstellensatz to obtain semidefinite programming certificates for the stability of the phase-difference system, which in turn implies synchronization of the original system. The results can be extended to any system of coupled oscillators where the forward-invariance on arcs can be established.
\end{abstract}
\vspace{10pt}

\section{Introduction}\label{sec:intro}

In 1665, Huygens observed that two pendulum clocks, weakly coupled through a heavy beam, eventually swung in opposite directions~\cite{willms2017huygens}. This phenomenon, now termed anti-phase synchronization, was the primary motivating factor for studying various synchronization issues in oscillators, such as phase synchronization and phase locking, among several others~\cite{arkady}. The fact that various objects in nature tend to seek harmony, a characteristic of synchronization, has only further fueled the research. This phenomenon can be observed in biological systems, such as groups of synchronously flashing fireflies~\cite{buck1988synchronous}, photosensitive neuron networks~\cite{hussain2021synchronization}, and the synchronization of neurons in memory processes~\cite{fell2011role}; and is a major feature studied in
superconducting Josephson junction~\cite{wiesenfeld1998frequency}, power grids~\cite{filatrella2008analysis} and cardiac conduction system~\cite{baleanu2021hyperchaotic}. An excellent survey on synchronization in oscillatory systems and complex networks can be found in~\cite{dorfler2014synchronization}. 

One of the most widely studied theoretical frameworks for understanding such collective behavior is the Kuramoto model.
In its classical form, the Kuramoto model comprises a set of phase oscillators coupled through a sinusoidal function~\cite{kuramoto_first}. Despite its simplicity, the models exhibit various properties, including phase/frequency synchronization~\cite{chopraspong}, and phase/frequency locking~\cite{wiesenfeld1998frequency}. Since then, Kuramoto models have been extended in a number of ways to account for more realistic features such as 
interconnecting topology~\cite{canale2008almost}, and general coupling functions~\cite{kudeyt2023certification}, among several others.

The following work explores Kuramoto models with a general coupling function that can be expressed in terms of a finite number of harmonics, each comprising sinusoidal terms. Deriving inspiration from synchronization on open half-circles for sinusoidal Kuramoto models~\cite{chopraspong,ha2010complete}, we will obtain certificates for local phase synchronization (on arcs) of the generalized Kuramoto model. 
Using the local LaSalle's invariance principle, the certificates are obtained as a semidefinite programming problem that can be solved using any modern convex program solver, for instance, our program \texttt{arcSOS-t}~\cite{Tripathi_Software_arcSOS-t-Solver_2025}. The program constructs a trigonometric polynomial that satisfies Lyapunov's time derivative condition on arcs and is inspired by our recent work on hypertoral systems~\cite{preprint_paper}. Subsequently, a domain of attraction is obtained using this construction and the invariance of the generalized Kuramoto models on certain arcs.

The paper is organized as follows. In Sect.~\ref{sec:notation}, we introduce some notations and definitions that will be used throughout the paper. In Sect.~\ref{sec:Kura_gen_coup}, we define the generalized Kuramoto model
. We reduce the dimension of this system by using phase-shift symmetry and introducing phase-difference variables in Sect.~\ref{sec:phase-difference}. We provide a consequence of LaSalle's invariance as Lemma~\ref{cor:Lasalles_cor_conf} for establishing local asymptotic stability of the phase-difference system on compact sets. In Sect.~\ref{sec:arc_invariance}, we define open/ closed arcs and obtain Theorem~\ref{thm:inv_arc} to establish forward-invariance of the model on arcs of certain lengths. This implies the forward-invariance of the phase-difference system on quotiented arcs. In Sect.~\ref{sec:SDP_cert}, we obtain an SDP certificate for the stability of the phase-difference system on quotiented arcs of certain lengths in Theorem~\ref{thm:local_stab_arcs_main}, implying local phase synchronization of the main model. We present some examples in Sect.~\ref{sec:example} and concluding remarks in Sect.~\ref{sec:conc}.

\section{Notations and Definitions}\label{sec:notation}

We will denote $\boldsymbol{0}$ (respectively, $\boldsymbol{1}$) as the vector of zeroes (respectively, ones), where the dimension will be clear with the context. The unit circle is denoted by $\mathbb{T}\vcentcolon=[0,2\pi)$ and is equipped with arc length metric, meaning ${\lvert\theta_1-\theta_2\rvert}$ is the length of the shortest arc joining $\theta_1$ and $\theta_2$, with a slight abuse of notation. For $\mathbf{k},\mathbf{n}\in\mathbb{Z}^d$, $\max\{\mathbf{k},\mathbf{n}\}=(\max\{k_1,n_1\},\ldots,\max\{k_d,n_d\})$; and the inequality $\mathbf{k}\le\mathbf{n}$ (respectively, ${\lvert\mathbf{k}\rvert\le\mathbf{n}}$) means $k_j\le n_j$ (respectively, $\lvert k_j\rvert\le n_j$) for all $j=1,\ldots,d$. The bar notation $\overline{z}$ denotes the complex conjugate if $z\in\mathbb{C}$ and the topological closure if $z$ is a set. 

A real-valued function is said to be a \textit{positive definite function} if it has a zero at the origin, and is positive elsewhere in the domain. A matrix $A$ is \textit{zero-sum} if all its entries sum up to zero. \textit{Nullity} is the dimension of the null space of a matrix. The shorthand $A\ge 0$ is used to denote that $A$ is positive semidefinite.

\section{Kuramoto Model with General Coupling}\label{sec:Kura_gen_coup}

Consider a class of generalized Kuramoto models of $d$ all-to-all coupled phase oscillators given by

\begin{equation}\label{eq:system_conf}
	\dot{\theta}_k=F_k(\boldsymbol{\theta})\vcentcolon=\omega+K \sum_{c=1}^d \left(\sum_{l=1}^L\alpha_l\sin{\left(l\left(\theta_c-\theta_k\right)+\beta_l\right)}\right), \quad k\in\{1,\dots, d\},   
\end{equation}
where $\boldsymbol\theta=(\theta_1,\theta_2,\dots,\theta_d)^\top\in \mathbb{T}^d$; $\omega$ is the natural frequency common to all phase oscillators; where $\alpha_l>0$ and $\lvert\beta_l\rvert<\pi/2$ for all $l=1,\ldots,L$. The model~\eqref{eq:system_conf} is a system on the hypertorus, $\mathbb{T}^d$, given by $\dot{\boldsymbol\theta}=\boldsymbol{F}(\boldsymbol{\theta})\vcentcolon=\left(F_1(\boldsymbol{\theta}),\ldots,F_d(\boldsymbol{\theta})\right)^\top$. The system~\eqref{eq:system_conf} exhibits \textit{local phase synchronization} in $\mathfrak{I}\subsetneq \mathbb{T}^d$ if for all $\boldsymbol{\theta}(0)\in\mathfrak{I}$, $\lim_{t\to\infty}\left\lvert\theta_i(t)-\theta_j(t)\right\rvert=0$ for all $i,j\in\{1,\ldots, d\}$. In this study, we will obtain semidefinite programming certificates for local phase synchronization of~\eqref{eq:system_conf}. 

\subsection{Phase-Difference System and LaSalle's Invariance Principle}\label{sec:phase-difference}

Due to the phase-shift symmetry of~\eqref{eq:system_conf}, which is the invariance of the system under transformation ${\boldsymbol{\theta}\mapsto \boldsymbol{\theta}+\epsilon\cdot \boldsymbol{1}}$, the dimension of the system can be reduced by $1$ to obtain a phase-difference system~\cite{kudeyt2023certification}. For instance, introducing phase-difference variables $\varphi_c=\theta_c-\theta_d$ for $c=1,\ldots,d-1$, we obtain a phase-difference system $\dot{\boldsymbol{\varphi}}=(\tilde{F}_1(\boldsymbol{\varphi}),\ldots,\tilde{F}_{d-1}(\boldsymbol{\varphi}))^\top\vcentcolon=\widetilde{\boldsymbol{F}}(\boldsymbol{\varphi})$, where $\boldsymbol{\varphi}=(\varphi_1,\ldots,\varphi_{d-1})\in\mathbb{T}^{d-1}$. We say that the phase-difference system is \textit{locally stable} in $\widetilde{\mathfrak{I}}\subsetneq\mathbb{T}^{d-1}$ if $\lim_{t\to\infty}\boldsymbol\varphi(t)=\boldsymbol{0}$ for all $\boldsymbol\varphi(0)\in \widetilde{\mathfrak{I}}$.

\begin{remark}\label{rem:dimension reduction}
	Consider the relation $\sim$ on $\mathbb{T}^d$ defined as $\boldsymbol{\theta}\sim \boldsymbol{\psi}$ if and only if $\boldsymbol{\theta}-\boldsymbol{\psi}=\varepsilon\cdot\boldsymbol{1}$ for some $\varepsilon\in\mathbb{R}$. Then $\sim$ is an equivalence relation. Note that the equivalence classes satisfy $[\boldsymbol{\theta}]=[\boldsymbol{\theta}-\theta_d\cdot \boldsymbol{1}]$, thus each class has a unique representation with last entry $0$. Thus, there is a canonical projection $\Phi_1\colon\mathbb{T}^d\to\left(\mathbb{T}^d/\sim\right)$ given by $\boldsymbol{\theta}\mapsto [\boldsymbol{\theta}-\theta_d\cdot \boldsymbol{1}]$, and an isomorphism $\Phi_2\colon \left(\mathbb{T}^d/\sim\right)\to \mathbb{T}^{d-1}$ given by $ [\boldsymbol{\theta}-\theta_d\cdot \boldsymbol{1}]\mapsto\left(\theta_1-\theta_d,\ldots,\theta_{d-1}-\theta_d\right)^\top$ . Denote $\left(\mathfrak{I}/\sim\right)=\Phi_2\circ \Phi_1 (\mathfrak{I})$.
\end{remark}

The local phase synchronization of~\eqref{eq:system_conf} in $\mathfrak{I}$ can be established by certifying the local stability of the phase-difference system in $\left(\mathfrak{I}/\sim\right)$, which can be established using the following consequence of LaSalle's Theorem~\cite[Theorem 4.4]{khalil2002nonlinear}.

\begin{lemma}\label{cor:Lasalles_cor_conf}
	If $\mathfrak{K}_{inv}\subseteq \mathbb{T}^{d-1}$ is a compact set that is positively invariant with respect to $\dot{\boldsymbol\varphi}=\widetilde{\boldsymbol{F}}(\boldsymbol{\varphi})$ and contains the origin. Let $V\colon \mathbb{T}^{d-1}\to \mathbb{R}$ be continuously differentiable function; such that $-\Grad V (\boldsymbol{\varphi})\cdot \widetilde{\boldsymbol{F}}(\boldsymbol{\varphi})$ is positive definite in $\mathfrak{K}_{inv}$; then every solution starting in $\mathfrak{K}_{inv}$ approaches to the origin.
\end{lemma}
\begin{proof}
	Since $\dot{V}(\boldsymbol{\varphi})=\Grad V (\boldsymbol{\varphi})\cdot \widetilde{\boldsymbol{F}}(\boldsymbol{\varphi})$ is negative definite  in the invariant set $\mathfrak{K}_{inv}$, $\{\boldsymbol{\varphi}\in\mathfrak{K}_{inv}\colon\ \dot{V}(\boldsymbol{\varphi})=\boldsymbol{0}\}=\{\boldsymbol{0}\}$, hence the result.
\end{proof}

For systems on a hypertorus, the construction of a Lyapunov-like function $V$ can be achieved using trigonometric polynomials~\cite{preprint_paper}, which will be briefly touched upon in Sect.~\ref{sec:SDP_cert}. Also, due to the structure of the generalized Kuramoto models~\eqref{eq:system_conf} considered in this study, invariant sets can be easily obtained. This is elaborated in the following section.

\section{Invariance of the Model on Arcs}\label{sec:arc_invariance}
\begin{definition}[Arc in $\mathbb{T}^d$,~\cite{dorfler2011critical}]
	An open arc of length $a<\pi$ in $\mathbb{T}^d$ is defined as the set of all points $\boldsymbol{\theta}\in\mathbb{T}^d$ whose components lie on a single open arc of length $a$ in $\mathbb{T}$, as depicted in Fig.~\ref{fig:arc_projection}, and is denoted by\begin{eqnarray*}
		\arc_d(a)&=&\{\boldsymbol{\theta}\in\mathbb{T}^d\colon \lvert\theta_i-\theta_j\rvert<a\ \text{ for all }i,j=1,\ldots,d\}.
	\end{eqnarray*}
	Define a closed arc of length $a<\pi$ in $\mathbb{T}^d$ as its topological closure $\overline{\arc_d(a)}$.
	\begin{figure}[h!]
		\includegraphics[scale=0.5]{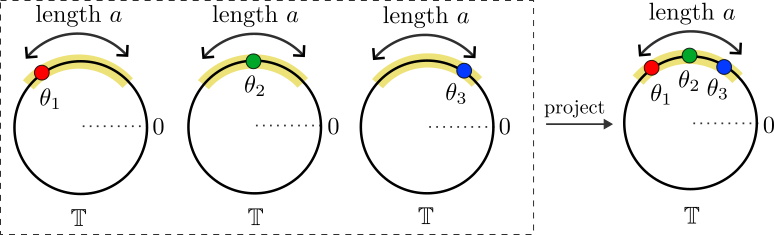}
		\centering
		\caption{The point $\boldsymbol{\theta}$ is in $Arc_d(a)$ if all components $\theta_j$ can be projected within an arc of length $a$ in $\mathbb{T}$.}\label{fig:arc_projection}
	\end{figure}
\end{definition}

\begin{theorem}\label{thm:inv_arc}
	Any arc $\overline{\arc_d(a)}$ with $a< \frac{2}{L}\left(\frac{\pi}{2}-\max_{l}\lvert\beta_l\rvert\right)$ is forward-invariant with respect to~\eqref{eq:system_conf}. Consequently, $\arc_d\left(\frac{\pi}{2}-\max_{l}\lvert\beta_l\rvert\right)$ is forward-invariant.
\end{theorem}

\begin{proof}
	First note that the sets $\mathcal{T}_{ij}=\{\boldsymbol{\theta}\in\mathbb{T}^d\colon \theta_i=\theta_j\}$ are invariant with respect to system~\eqref{eq:system_conf} since $\boldsymbol{F}(\boldsymbol{\theta})\in\mathcal{T}_{ij}$ for any $\boldsymbol\theta\in\mathcal{T}_{ij}$. Thus given a pair $(i,j)$ and a solution $\boldsymbol{\theta}(t)$, either $\theta_i(t)$ and $\theta_j(t)$ do not collide or they are identically equal. Let $a<\frac{2}{L}\left(\frac{\pi}{2}-\max_{l}\lvert\beta_l\rvert\right)$. We claim that the set $\overline{\mathrm{Arc}_d\left(a\right)}$ is forward-invariant under the flow. Let $\boldsymbol{\theta}(0)\in\overline{\mathrm{Arc}_d\left(a\right)}$, then $\theta_c(0)$, $c=1,\ldots,d$, can be plotted on a closed arc of length $a<\pi$. Choose $m, M \in\{1,\ldots,d\}$ to be indices of $\theta_c(0)$, occurring first and last (may not be unique), respectively, when we travel counter-clockwise on the arc. For instance, in Fig.~\ref{fig:arc_projection}, $d=3$, $m=3$ and $M=1$. Due to invariance of $\mathcal{T}_{ij}$, each $\theta_c(t)$ lies between $\theta_m(t)$ and $\theta_M(t)$ for all $t\ge 0$. Set the convention, $\textup{D}(\boldsymbol{\theta}(t))\vcentcolon=\theta_{M}(t)-\theta_{m}(t)>0$. Then we have  \begin{eqnarray*}
		\nonumber \dot{\textup{D}}(\boldsymbol{\theta}(t))&=& K \sum_{c=1}^d\sum_{l=1}^L\alpha_l\,\biggl[\sin{\left(l\left(\theta_c(t)-\theta_M(t)\right)+\beta_l\right)}-\sin{\left(l\left(\theta_c(t)-\theta_m(t)\right)+\beta_l\right)}\biggr]\\
		&=&\nonumber-2 K \sum_{l=1}^L\alpha_l\Biggl[\sin\biggl(l\,\left(\theta_{M}(t)-\theta_{m}(t)\right)\biggr)\,\cos(\beta_l) + \sin\left(l\,\frac{\theta_{M}(t)-\theta_{m}(t)}{2}\right)\Biggr.\\
		& &\hspace{80pt}\Biggl.\times \sum_{c\notin \{m,M\}} \,\cos\left(l\left(\theta_c(t)-\frac{\theta_{M}(t)+\theta_{m}(t)}{2}\right)+\beta_l\right) 
		\Biggr].
	\end{eqnarray*}
	If $\boldsymbol{\theta}(t)\in\overline{\mathrm{Arc}_d\left(a\right)}$, then $l\,\left(\theta_{M}(t)-\theta_{m}(t)\right)\le La<\pi$; $\lvert\beta_l\rvert<\pi/2$; and \begin{eqnarray*}
		\nonumber
		\left\lvert l\left(\theta_c(t)-\frac{\theta_{M}(t)+\theta_{m}(t)}{2}\right)+\beta_l\right\rvert&\le& L\left(\left\lvert \frac{\theta_c(t)-\theta_{M}(t)}{2}\right\rvert+\left\lvert\frac{\theta_c(t)-\theta_{m}(t)}{2}\right\rvert\right)+\max_l \lvert\beta_l\rvert\\
		\nonumber &=& L\left( \frac{\theta_{M}(t)-\theta_c(t)}{2}+\frac{\theta_c(t)-\theta_{m}(t)}{2}\right)+\max_l \lvert\beta_l\rvert\\
		&\le& L\frac{a}{2}+\max_l \lvert\beta_l\rvert< \frac{\pi}{2}.
	\end{eqnarray*}
	We conclude $\theta_{M}(t)-\theta_{m}(t)$ is nonincreasing  for all $t\ge 0$, thus $\boldsymbol{\theta}(t)\in\overline{\arc_d(a)}$.
\end{proof}

\begin{corollary}\label{cor:invariant_set_PD}
	For the phase-difference system $\dot{\boldsymbol{\varphi}}=\widetilde{\boldsymbol{F}}(\boldsymbol{\varphi})$ corresponding to system~\eqref{eq:system_conf}, the set $\left(\overline{\arc_d(a)}/\sim\right)$ is forward-invariant for all $a< \frac{2}{L}\left(\frac{\pi}{2}-\max_{l}\lvert\beta_l\rvert\right)$, where the notation $(\cdot/\sim)$ is from Remark~\ref{rem:dimension reduction}. 
\end{corollary}
\begin{proof}
	The proof is straightforward and has been omitted. Moreover,
	\begin{eqnarray*}
		\left(\overline{\arc_d(a)}/\sim\right)=\{\boldsymbol{\varphi}\in\mathbb{T}^{d-1}\colon\ \ \lvert \varphi_i-\varphi_j\rvert\le a,\ \lvert\varphi_i\rvert \le a\ \ \ \forall\,i,j\}\subset \overline{\arc_{d-1}(a)}.
	\end{eqnarray*}
\end{proof}

\section{SDP Certificates for Local Stability of the Phase-Difference System}\label{sec:SDP_cert}

In this section, we briefly review the theory of trigonometric polynomials~\cite{dumitrescu2007positive}, which will be essential in obtaining local SDP certificates. A trigonometric polynomial $R(\boldsymbol{\varphi}):\mathbb{T}^{d-1}\to \mathbb{R}$ is defined by
\begin{equation*}
	R(\boldsymbol{\varphi})=\sum_{\mathbf{k}=-\mathbf{n_r}}^{\mathbf{n_r}}r_\mathbf{k}\, {\rm{e}}^{\mathrm{i}\,\mathbf{k}\cdot\boldsymbol{\varphi}}\vcentcolon=\sum_{k_1=-n_r(1)}^{n_r(1)}\ldots \sum_{k_{d-1}=-n_r(d-1)}^{n_r(d-1)}\,r_\mathbf{k}\, {\rm{e}}^{\mathrm{i}\,\mathbf{k}\cdot\boldsymbol{\varphi}},
\end{equation*}
where $\mathbf{n_r}=\left(n_r(1),n_r(2)\ldots,n_r(d-1)\right)\in \mathbb{Z}_{\ge 0}^{d-1}$ is the number of harmonics; $\mathbf{k}=(k_1,k_2,\dots,k_{d-1})\in\mathbb{Z}^{d-1}$ is an index satisfying $\lvert\mathbf{k}\rvert\le\mathbf{n_r}$; $r_\mathbf{k}$'s are complex Fourier coefficients of $R(\boldsymbol{\theta})$ satisfying $r^{}_{-\mathbf{k}}=\overline{r_\mathbf{k}}$; $\mathrm{i}$ is the imaginary unit; ${\mathbf{k}\cdot\boldsymbol\varphi=\sum_{c=1}^{d-1} k_c\varphi_c}$; and $\mathbf{n_r}\in \mathbb{Z}_{\ge 0}^d$ is taken as minimal, meaning $\mathbf{n_r}$ is the smallest index such that $r_{\mathbf{k}}= 0$ for all $\lvert\mathbf{k}\rvert> \mathbf{n_r}$, and called the degree of the trigonometric polynomial. 

Given a trigonometric polynomial $R(\boldsymbol\varphi)$, there exists a Hermitian matrix $G_R$ of size $\prod_{c=1}^{d-1} \left(n_r(c)+1\right)$ satisfying
\begin{equation}\label{eq: GramMatrix}
	R(\boldsymbol\varphi)=z(\boldsymbol\varphi)^\dagger\, G_R\, z(\boldsymbol\varphi),
\end{equation}
where $z(\boldsymbol\varphi)$ is a vector of \textit{primitive} (having coefficient equal to $1$) monomials that form a trigonometric basis for $\mathbb{T}^{d-1}$,~\cite{dumitrescu2007positive}. Such a matrix $G_R$ is said to be a \textit{Gram matrix representation} of $R(\boldsymbol{\varphi})$. Any matrix $G_R$ satisfying~\eqref{eq: GramMatrix} is positive semidefinite if and only if $R(\boldsymbol{\varphi})\ge 0$,~\cite{dumitrescu2007positive}. For all $\lvert\mathbf{k}\rvert\le \mathbf{n_r}$, any Gram matrix representation associated with $R(\boldsymbol{\varphi})$ satisfies \begin{equation}\label{eq:traceparametrization}
	r_\textbf{k}=\trace\left[\left(T_{k_{d-1}}^{n_r(d-1)+1}\otimes T_{k_{d-2}}^{n_r(d-2)+1}\otimes \cdots \otimes T_{k_1}^{n_r(1)+1}\right)^\dagger G_R\right]\vcentcolon=\ct{k}{n_r}{G_R},
\end{equation}
where $T_k^n$ is the $k^{th}$ elementary Toeplitz matrix of size $n$, meaning it is a $(0,1)-$matrix with ones only on the $k^{th}$ diagonal, that is, $T_{(i,j)}=1$ if and only if $j-i=k$. This is known as the trace parametrization property~\cite{dumitrescu2007positive} of trigonometric polynomials. 

Given trigonometric polynomials $R(\boldsymbol{\varphi})$ and $P(\boldsymbol{\varphi})$, the derivatives $D_cR(\boldsymbol{\varphi})\vcentcolon=\frac{\mathrm{d}R}{\mathrm{d}\varphi_c}(\boldsymbol{\varphi})$ for ${c=1,\ldots,d-1}$; the sum $(P+R)(\boldsymbol{\varphi})\vcentcolon=P(\boldsymbol{\varphi})+R(\boldsymbol{\varphi})$; and product $PR(\boldsymbol{\varphi})\vcentcolon=P(\boldsymbol{\varphi})R(\boldsymbol{\varphi})$ are trigonometric polynomials of degree $\mathbf{n_r}$, $\mathbf{n}_{sum}\le\max\{\mathbf{n_p},\mathbf{n_r}\}$ and $\mathbf{n_p+n_r}$, respectively. Trace parametrization allows us to obtain relations concerning their Gram matrix representations~\cite{preprint_paper} as\begin{eqnarray}\label{eq:deriv-Gram}
	\ct{k}{n_r}{G_{D_c R}}
	&=&\mathrm{i} \, k_c\,\ct{k}{n_r}{G_R},\\
	\label{eq:sum-Gram}\ct{k}{\max\{n_p,n_r\}}{G_{P+R}}&=&\ct{k}{n_p}{G_p}+\ct{k}{n_r}{G_R},\\
	\label{eq:product-Gram}
	\ct{k}{n_{p}+n_{r}}{G_{PR}}&=&
	\sum_{\lvert\mathbf{j}\rvert\le\mathbf{n_{p}+n_{r}}}\ct{k-j}{n_p}{G_P}\, \ct{j}{n_r}{G_R}. 
\end{eqnarray}

The system $\dot{\boldsymbol{\theta}}=\boldsymbol{F}(\boldsymbol{\theta})$ given in~\eqref{eq:system_conf} clearly has a trigonometric series expansion of degree $\boldsymbol{L}=L\times\mathbf{1}$ for each of its components $F_c\colon\mathbb{T}^d\to\mathbb{R}$. Consequently, the components of corresponding phase-difference system $\widetilde{F}_c:\mathbb{T}^{d-1}\to \mathbb{R}$ with variables $\varphi_c=\theta_c-\theta_d$ for $c=1,\ldots,d-1$ also have a trigonometric polynomial expansion of degree $L\cdot\mathbf{1}$,
\begin{equation}\label{eq:Vector_field_fourier_PD_conf}
	\widetilde{F}_c(\boldsymbol\varphi)=\sum_{\mathbf{k}=-\boldsymbol{L}}^{\boldsymbol{L}}\widetilde{f}_\mathbf{k}^{\,(c)}\, e^{\mathrm{i}\,\mathbf{k}\cdot\boldsymbol\varphi}\vcentcolon=\sum_{k_1=-L}^{L}\ldots \sum_{k_{d-1}=-L}^{L}\,\widetilde{f}_\mathbf{k}^{\,(c)}\, {\rm{e}}^{\mathrm{i}\, \mathbf{k}\cdot\boldsymbol{\varphi}}.
\end{equation}
We will obtain a trigonometric polynomial $V(\boldsymbol{\varphi})$, such that $\Grad V (\boldsymbol{\varphi})\cdot \widetilde{\boldsymbol{F}}(\boldsymbol{\varphi})$ is negative definite on closed invariant arcs, to establish local stability of ${\dot{\boldsymbol{\varphi}}=\widetilde{\boldsymbol{F}}(\boldsymbol{\varphi})}$ using Lemma~\ref{cor:Lasalles_cor_conf}. For the same, we need to characterize the positivity of the trigonometric polynomial $W(\boldsymbol{\varphi})$ on certain arcs. For $r\in\mathbb{N}$, the set $\overline{\arc_{d-1}(\pi/2r)}$ can be expressed as\begin{eqnarray}\label{eq:domain_r}
	\mathcal{D}_{r} &=&\left\{\boldsymbol{\varphi}\in\mathbb{T}^{d-1}\colon\, Q_r^{(p,q)}(\boldsymbol{\varphi})\vcentcolon=\cos\left((\varphi_p-\varphi_q)r\right)\ge 0,\ 1\le p,q\le d-1\right\}.
\end{eqnarray}
Note that each $Q_r^{(p,q)}(\boldsymbol{\varphi})$ has degree $\le \boldsymbol{r}\vcentcolon=r\times\mathbf{1}$, and satisfies \begin{equation}\label{eq:trace_parameter_Hpqr}
	\ct{m}{\boldsymbol{r}}{G_{Q_r^{(p,q)}}}=\begin{cases}
		\begin{aligned}
			&1/2, &\text{ if } \mathbf{m}=\pm\mathfrak{S}_r^{(p,q)}\\
			&0, &\text{ otherwise}
		\end{aligned}
	\end{cases},
\end{equation} where $\mathfrak{S}_r^{(p,q)}=(0,\ldots,0,\underset{p^{th}\text{ element}}{r},0,\ldots,0,\underset{q^{th}\text{ element}}{-r},0,\ldots,0)$. Then Putinar's Positivstellensatz~\cite{dumitrescu2007positive} leads to the following result.

\begin{lemma}[Putinar's Positivstellensatz~\cite{dumitrescu2007positive}]\label{lem:Putinar}
	Consider trigonometric polynomials $Q_r^{(p,q)}(\boldsymbol{\varphi})\vcentcolon=\cos\left((\varphi_p-\varphi_q)r\right)$ for $1\le p\ne q\le d-1$, and $\mathcal{D}_r$ be as defined in~\eqref{eq:domain_r}. Then, a trigonometric polynomial $W(\boldsymbol\varphi)$ is positive on $\mathcal{D}$ (that is, $W(\boldsymbol\varphi)>0$ for all $\boldsymbol\varphi\in\mathcal{D}$), then there exist sum of squares polynomials $S_{0}^w(\boldsymbol\varphi)$, $S_{(p,q)}^w(\boldsymbol\varphi)$ for $1\le p\ne q\le d-1$ such that \begin{equation*}
		W(\boldsymbol\varphi)=S_0^w(\boldsymbol\varphi)+\sum_{1\le p\ne q\le d-1} S_{(p,q)}^w(\boldsymbol\varphi)\, Q_r^{(p,q)}(\boldsymbol\varphi).
	\end{equation*}
\end{lemma}

We are now ready to obtain the main result of this paper.

\begin{theorem}[Local stability on arcs]\label{thm:local_stab_arcs_main}
	Given $r\ge 1$, $\mathbf{n_v}\ge\boldsymbol{r}-\boldsymbol{L}$, and phase-difference system $\dot{\boldsymbol{\varphi}}=\widetilde{\boldsymbol{F}}(\boldsymbol{\varphi})$ corresponding to system~\eqref{eq:system_conf} having trigonometric polynomial expansion~\eqref{eq:Vector_field_fourier_PD_conf}, if there exists a Hermitian matrix $G_V$ of size $\prod_{c=1}^{d-1}\left(n_v{(c)}+1\right)$; a zero-sum Hermitian matrix $G_{S_0^w}\ge 0$ of size $\prod_{c=1}^{d-1}\left(n_v{(c)}+L+1\right)$; and zero-sum Hermitian matrices $G_{S_{(p,q)}^w}\ge 0,\ 1\le p\ne q\le d-1$, of size $\prod_{c=1}^{d-1}\left(n_v(c)-r+L+1\right)$, such that at least one $G_{S_{\ell}^w}$ has nullity $1$, and \begin{eqnarray}\label{eq:main_sdp_problem}
		\ct{k}{n_v+\textit{\textbf{L}}}{G_{S_0^w}}+\frac{1}{2}\sum_{\substack{p,q\in\{1,\ldots,d-1\}\\p\ne q}}\sum_{\boldsymbol{j}=\boldsymbol{k}\pm \mathfrak{S}_r^{(p,q)}}\ct{j}{n_v+\textit{\textbf{L}}-\textit{\textbf{r}}}{G_{S_{(p,q)}^w}}=-\sum_{c=1}^{d-1}\sum_{\lvert\boldsymbol{j}\rvert\le\boldsymbol{n_v}+\textit{\textbf{L}}}\,\mathrm{i}\,j_c\, \ct{j}{n_v}{G_V}\, \widetilde{f}_{\boldsymbol{k-j}}^{\,(c)},
	\end{eqnarray}
	for all $\lvert\mathbf{k}\rvert\le \boldsymbol{n_v}+\boldsymbol{L}$, then the system~\eqref{eq:system_conf} exhibits local phase synchronization in $\arc_d\left(\min\left\{\frac{\pi}{2r},\frac{2}{L}\left(\frac{\pi}{2}-\max_{l}\lvert\beta_l\rvert\right)\right\}\right)$.
\end{theorem}

\begin{proof}
	Using~\eqref{eq:trace_parameter_Hpqr}, Eq.~\eqref{eq:main_sdp_problem} can be rewritten as
    \begin{eqnarray*}
		\nonumber\ct{k}{n_v+\textit{\textbf{L}}}{G_{S_0^w}}+\sum_{1,\le p\ne q \le d-1}\sum_{\lvert\boldsymbol{j}\rvert\le \boldsymbol{n_v}+\textit{\textbf{L}}}\ct{j}{n_v+\textit{\textbf{L}}-\textit{\textbf{r}}}{G_{S_{(p,q)}^w}}\,\ct{k-j}{\textit{\textbf{r}}}{G_{Q_\ell}}=-\hspace{-10pt}\sum_{\substack{\lvert\boldsymbol{j}\rvert\le\boldsymbol{n_v}+\textit{\textbf{L}}\\c=1,\ldots,d-1}}\hspace{-15pt}\,\mathrm{i}\,j_c\, \ct{j}{n_v}{G_V}\, \widetilde{f}_{\boldsymbol{k-j}}^{\,(c)}.
	\end{eqnarray*}
    Using equations~$\eqref{eq:traceparametrization}-~\eqref{eq:product-Gram}$ and linearity of trace, we can rewrite the equality for any $\lvert\mathbf{k}\rvert\le \boldsymbol{n_v}+\textit{\textbf{L}}$ as
    \begin{eqnarray*}
		\ct{k}{n_v+\textit{\textbf{L}}}{G_{S_0^w}}+\sum_{\substack{p,q=1\\p\ne q}}^{d-1}\ct{k}{n_v+\textit{\textbf{L}}}{G_{S_{(p,q)}^w\,Q_{(p,q)}}}&=&-\hspace{-10pt}\sum_{\substack{\lvert\boldsymbol{j}\rvert\le\boldsymbol{n_v}+\textit{\textbf{L}}\\c=1,\ldots,d-1}}\hspace{-15pt}\ct{j}{n_v}{G_{D_cV}}\,\ct{k-j}{\textit{\textbf{L}}}{G_{\widetilde{F}_c}}\\
		\iff\ct{k}{n_v+\textit{\textbf{L}}}{G_{S_0^w}+\sum_{\substack{p,q=1\\p\ne q}}^{d-1}G_{S_{(p,q)}^w\,Q_{(p,q)}}}&=&-\sum_{c=1}^{d-1}\ct{k}{n_v+\textit{\textbf{L}}}{ G_{D_cV\, \widetilde{F}_c}},\\
		\iff\ct{k}{n_v+\textit{\textbf{L}}}{G_{S_0^w+\sum_{1\le p\ne q\le d-1}S_{(p,q)}^w\,Q_{(p,q)}}}&=&\ct{k}{n_v+\textit{\textbf{L}}}{ G_{-\sum_{c=1}^{d-1}D_cV\, \widetilde{F}_c}},
	\end{eqnarray*}
	which means that the Gram representation matrices on the left and right correspond to the same trigonometric polynomial. Hence, \begin{eqnarray*}
		W(\boldsymbol{\varphi})&\vcentcolon=&S_0^w(\boldsymbol\varphi)+\sum_{1\le p\ne q\le d-1} S_{(p,q)}^w(\boldsymbol\varphi)\, Q_r^{(p,q)}(\boldsymbol\varphi)
        =
        -\Grad V(\boldsymbol{\varphi})\cdot \widetilde{\boldsymbol{F}}(\boldsymbol{\varphi}).
	\end{eqnarray*}
	By hypothesis $G_{S_{(p,q)}^w}\ge 0$, hence   $S_{(p,q)}^w(\boldsymbol{\varphi})\ge 0$. Also, $S_{(p,q)}^w(\boldsymbol{\varphi})=z(\boldsymbol{\varphi})^\dagger\, G_{S_{(p,q)}^w}\, z(\boldsymbol{\varphi})$ by~\eqref{eq: GramMatrix}. Since $z(\boldsymbol{0})=\boldsymbol{1}$, the zero-sum condition implies $S_{(p,q)}^w(\boldsymbol{0})=0$. Since $G_{S_{\ell}}$ has nullity $1$ for $\ell=0$ or $\ell=(p_0,q_0)$, we have $S_{\ell}^w(\boldsymbol{\varphi})>0$ for all $\boldsymbol{\varphi}\in \mathcal{D}\setminus\{\boldsymbol{0}\}$. Hence, existence of positive definite ${W(\boldsymbol{\varphi})=-\Grad V(\boldsymbol{\varphi})\cdot \widetilde{\boldsymbol{F}}(\boldsymbol{\varphi})}$ on $\mathcal{D}_r$ is established. Now, $D_r\supseteq(\overline{\arc_d(\pi/2r)}/\sim)$. Also, $(\overline{\arc_d(a_\ast)}/\sim)$ is forward-invariant for all $a_\ast< \frac{2}{L}\left(\frac{\pi}{2}-\max_{l}\lvert\beta_l\rvert\right)$ , by Corollary~\ref{cor:invariant_set_PD}. Thus, the hypothesis of Lemma~\ref{cor:Lasalles_cor_conf} is satisfied taking $\mathfrak{K}_{inv}=(\overline{\arc_d(\min\{\pi/2r,a_\ast\})}/\sim)$. Hence, every solution starting in $\mathfrak{K}_{inv}$ approaches the origin. This implies local phase synchronization of system~\eqref{eq:system_conf} in $\mathcal{I}=\overline{\arc_d(\min\{\pi/2r,a_\ast\})}$ for any $a_\ast< \frac{2}{L}\left(\frac{\pi}{2}-\max_{l}\lvert\beta_l\rvert\right)$. 
\end{proof}

\begin{remark}\label{rem:arsos_and_nullity}
    To check the feasibility of the SDP obtained in Theorem~\ref{thm:local_stab_arcs_main}, the authors developed a dedicated program based on SeDuMi in MATLAB. It is called \texttt{arcSOS-t} solver and is available on GitHub~\cite{Tripathi_Software_arcSOS-t-Solver_2025}. The nullity condition in Theorem~\ref{thm:local_stab_arcs_main} is nonconvex and can not be directly imposed in an SDP. The issue can, however, be bypassed by imposing conditions $G_{S_{\ell}^w}\, \boldsymbol{1}=0$ (implies $G_{S_{\ell}^w}$ is zero-sum) and $\lambda_{min}(G_{S_{\ell}^w}+\boldsymbol{1}\,\boldsymbol{1}^\top)\ge 0.001$. 
    
    Any positive semidefinite matrix $A$ satisfying condition $A\,\boldsymbol{1}=0$ increases in rank by $1$ on addition by $\boldsymbol{1}\,\boldsymbol{1}^\top$. This can be proved by fixing an orthogonal basis (containing $\boldsymbol{1}$) that diagonalizes $A$, and observing the images of the basis vectors under $A+\boldsymbol{1}\,\boldsymbol{1}^\top$. It follows that $A+\boldsymbol{1}\,\boldsymbol{1}^\top$ shares all eigenvalues with $A$, except the one corresponding to the eigenvector $\boldsymbol{1}$, which becomes positive on this perturbation. Thus, the matrix $G_{S_{\ell}^w}$ satisfying $G_{S_{\ell}^w}\,\boldsymbol{1}=0$ is nullity $1$ if and only if $G_{S_{\ell}^w}+\boldsymbol{1}\,\boldsymbol{1}^\top$ is full rank. Now the second condition imposes the nullity-one condition on $G_{S_{\ell}^w}$ (since $G_{S_{\ell}^w}+\boldsymbol{1}\,\boldsymbol{1}^\top$ is positive definite matrix and thus has full rank), and also bounds the positive eigenvalues of $G_{S_{\ell}^w}$ by $0.001$ from below. 
\end{remark}

\section{Examples}\label{sec:example}

We present two examples with three oscillators ($d=3$) and $L=1,2$ harmonics in $g$. Define ${g(x)=\sum_{l=1}^L\alpha_l\,\sin(l\,x+\beta_l)}$. Then the system~\eqref{eq:system_conf} is rewritten as \begin{equation*}
	\dot{\theta}_k=F_k(\boldsymbol{\theta})\vcentcolon=\omega+K \sum_{c=1}^3 g(\theta_c-\theta_k), \quad k\in\{1,2,3\}.   
\end{equation*}
Defining $\varphi_1=\theta_1-\theta_3$ and $\varphi_2=\theta_2-\theta_3$, we obtain the phase-difference system as\begin{eqnarray*}\label{eq:example_phase_dif}
	\begin{aligned}
		\dot{\varphi_1}&=&\widetilde{F}_1(\varphi_1,\varphi_2):=K\,(g(\varphi_2-\varphi_1)+g(-\varphi_1)-g(\varphi_1)-g(\varphi_2)),\\
		\dot\varphi_2&=&\widetilde{F}_2(\varphi_1,\varphi_2):=K\,(g(\varphi_1-\varphi_2)+g(-\varphi_2)-g(\varphi_1)-g(\varphi_2)),
	\end{aligned}
\end{eqnarray*}
whose stability is independent of $K$. Hence, we can substitute $K=1$. The coefficients for the phase-difference vector field are then given in Table~\ref{tab:PD_coefficients}. In the examples that follow, we will utilize our program \texttt{arcSOS-t}~\cite{Tripathi_Software_arcSOS-t-Solver_2025}, which was discussed in Remark~\ref{rem:arsos_and_nullity}, to check the feasibility of the SDP problem~\eqref{eq:main_sdp_problem}.

\begin{table}[h!]
	\setlength{\tabcolsep}{12pt}\centering
		\caption{Trigonometric polynomial coefficients for $\dot{\boldsymbol{\varphi}}=\widetilde{F}(\boldsymbol{\varphi})$ in Sect.~\ref{sec:example}}
		\label{tab:PD_coefficients}
        \begin{tabular}{ccc}
			\hline $\boldsymbol{k}$ & $\widetilde{f}_{\boldsymbol{k}}^{\,(1)}$ & $\widetilde{f}_{\boldsymbol{k}}^{\,(2)}$ \\[5pt]
			$(-l,l)\colon l=1,\ldots,L$ & $-\mathrm{i}\,\alpha_l\,{\rm{e}}^{\mathrm{i}\,\beta_l}/2$ & $\mathrm{i}\,\alpha_l\,{\rm{e}}^{-\mathrm{i}\,\beta_l}/2$\\
			$(l,0)\colon l=1,\ldots,L$ & $\mathrm{i}\,\alpha_l\,\cos(\beta_l)$ & $\mathrm{i}\,\alpha_l\,{\rm{e}}^{\mathrm{i}\,\beta_l}/2$\\
			$(0,l)\colon l=1,\ldots,L$ & $\mathrm{i}\,\alpha_l\,{\rm{e}}^{\mathrm{i}\,\beta_l}/2$ & $\mathrm{i}\,\alpha_l\,\cos(\beta_l)$\\
			otherwise & 0 & 0\\ \hline
		\end{tabular}
\end{table}

\begin{example}
	For $g(x)=4\,\sin(x+\pi/8)$. The SDP~\eqref{eq:main_sdp_problem} is feasible for $r=1$, $\mathbf{n_v}=(2,2)$, as per \textup{\texttt{arcSOS-t}}. The system exhibits local phase synchronization on $\arc_3(\pi/2)$ by Theorem~\ref{thm:local_stab_arcs_main}.
\end{example}

\begin{example}
	For $g(x)=4\,\sin(x+\pi/8)+2\,\sin(2x-\pi/4)$. The SDP~\eqref{eq:main_sdp_problem} is feasible for $r=1$, $\mathbf{n_v}=(2,2)$, as per \textup{\texttt{arcSOS-t}}, and local phase synchronization on $\arc_3(\pi/4)$ is established by Theorem~\ref{thm:local_stab_arcs_main}.
\end{example}

\section{Concluding Remarks}\label{sec:conc}

In this work, we obtained SDP certificates for establishing local phase synchronization on certain arcs for Kuramoto models whose coupling function can be expressed as a sum of a finite number of harmonics, each comprising sinusoidal terms. Using invariance of the system on arcs of certain lengths (Theorem~\ref{thm:inv_arc}), and the existence of a Lyapunov-like function on Putinar-inspired domains, we obtained Theorem~\ref{thm:local_stab_arcs_main} to establish local phase synchronization of the generalized Kuramoto model on arcs. The feasibility of the conditions obtained in the theorem is checked for some examples using our program~\texttt{arcSOS-t}~\cite{Tripathi_Software_arcSOS-t-Solver_2025}.

\section*{Acknowledgements}
This work was supported by the TUBITAK 1001 Research fund [grant number: 122E522].

%
%
\bibliographystyle{plain}
\bibliography{mybibfile}
\end{document}